\documentclass[reqno, 12pt, reqno]{amsart}
\usepackage{amsfonts, amsthm, amsmath, amssymb}
\usepackage{hyperref}
\usepackage{pstricks}
\usepackage{pstricks-add}

\RequirePackage{mathrsfs} \let\mathcal\mathscr

\numberwithin{equation}{section}

\newtheorem{theorem}{Theorem}[section]
\newtheorem{lemma}[theorem]{Lemma}

\theoremstyle{definition}
\newtheorem*{ack}{Acknowledgements}

\newtheorem*{rem}{Remark}

\usepackage{color}
\definecolor{red}{rgb}{1,0,0}
\definecolor{blue}{rgb}{.2,.6,.75}
\definecolor{green}{rgb}{.4,.7,.4}

\renewcommand{\d}{\mathrm{d}}
\renewcommand{\phi}{\varphi}

\newcommand{\PP}{\mathbb{P}}

\newcommand{\FF}{\mathbb{F}}
\newcommand{\ZZ}{\mathbb{Z}}

\newcommand{\NN}{\mathbb{N}}
\newcommand{\QQ}{\mathbb{Q}}
\newcommand{\RR}{\mathbb{R}}

\renewcommand{\leq}{\leqslant}

\renewcommand{\geq}{\geqslant}

\definecolor{lightgray}{gray}{0.85}

\usepackage{pst-pdgr, pstricks, graphicx}%
\usepackage{auto-pst-pdf} 
\usepackage{caption}

\DeclareMathOperator{\Mod}{mod} 
\renewcommand{\bmod}[1]{\,(\Mod{#1})}

\newcommand{\eeq}{\end{equation}}
\newcommand{\beql}[1]{\begin{equation}\label{#1}}

\address{%
R\'egis de la Bret\`eche
\\
Institut de Math\'ematiques de Jussieu, UMR 7586
\\
Universit\'e Paris-Diderot
\\
UFR de Math\'ematiques, case 7012
\\
B\^atiment Sophie Germain
\\
75205 Paris Cedex 13, France} 
\email{regis.de-la-breteche@imj-prg.fr}

\address{%
Kevin Destagnol
\\
IST Austria
\\
Am Campus 1
\\
3400 Klosterneuburg, Austria} 
\email{kevin.destagnol@ist.ac.at}

\address{%
Jianya Liu
\\
School of Mathematics
\\
Shandong University
\\
Jinan
\\
Shandong 250100
\\
China} \email{jyliu@sdu.edu.cn}

\address{%
Jie Wu\\
CNRS, UMR 8050\\
Laboratoire d'Analyse et de Math\'ematiques Appliqu\'ees\\
Universit\'e Paris-Est Cr\'eteil\\
61 Avenue du G\'en\'eral de Gaulle\\
94010 Cr\'eteil cedex\\
France
}
\email{jie.wu@math.cnrs.fr}

\address{%
Yongqiang Zhao
\\
Westlake University
\\
Hangzhou
\\
Zhejiang 310024    
\\
China}
\email{yzhao@wias.org.cn}

\begin{document}

\title{On a certain non-split  cubic surface} 
\author{R. de la Bret\`eche, K. Destagnol, J. Liu, J. Wu \& Y. Zhao}
\maketitle

\begin{abstract} 
In this note, we establish an asymptotic formula with a power-saving error term for the number of rational points of bounded height on the 
singular cubic surface of $\mathbb{P}^3_{\QQ}$ 
$$
x_0(x_1^2 +  x_2^2)=x_3^3
$$  
in agreement with the Manin-Peyre conjectures. 

\end{abstract}

\section{Introduction and results}

Let $V\subset \PP_{\QQ}^3$ be the cubic surface defined by 
$$ 
x_0(x_1^2+x_2^2)-x_3^3=0. 
$$
The surface $V$ has three singular points $\xi_1=[1:0:0:0]$,  $\xi_2=[0:1:i:0]$ 
and $\xi_3=[0:1:-i:0]$.
It is easy to see that the only three lines contained in    $V_{\overline{\QQ}}=V\times_{{\rm{Spec}}(\QQ)}{\rm{Spec}}(\overline{\QQ})$ are
$$\ell_1:=\{x_3=x_1-ix_2=0\},\quad
\ell_2:=\{x_3=x_1+ix_2=0\},
$$  
and $$\ell_3:=\{x_3=x_0=0\}.$$
Clearly both $\ell_1$ and $\ell_2$ pass through $\xi_1$, which is actually the only rational point lying on these two lines. 

Let $U=V\smallsetminus  \{\ell_1\cup\ell_2\cup \ell_3\}$, 
and $B$ a parameter that can approach infinity. 
In this note we are concerned with the behavior of the counting function
$$
N_U( B)=\#\{{\bf x}\in U(\QQ): H({\bf x})\leq B\},
$$
where $H$ is the anticanonical height function on  $V$ defined by
\begin{equation}
H({\bf x}):=\max\Big\{ |x_0|,\, \sqrt{x_1^2+x_2^2},\,|x_3|\Big\}\label{defH}
\end{equation}
where each $x_j\in \ZZ$ and $\gcd(x_0, x_1, x_2, x_3)=1$. 
The main result of this note is the following. 

\begin{theorem}\label{t:main}
There exists a constant $\vartheta>0$ and a polynomial 
$Q\in \RR[X]$ of degree $3$ such that  
\begin{equation}\label{N_UB=}
N_U(B)=
BQ(\log B)+O(B^{1-\vartheta}).
\end{equation}
The leading coefficient $C$ of $Q$ satisfies 
\begin{equation}\label{TmC=}
C= \frac{7}{216}(3\pi)\bigg(\frac{\pi}{4}\bigg)^3 \tau 
\end{equation}
with 
$$
\tau = \prod_{p }\bigg( 1-\frac{1}{p} \bigg)^4\bigg( 1-\frac{\chi(p)}{p} \bigg)^3  
\bigg(1+\frac{2+3\chi(p)+2\chi^2(p)}{p}+\frac{\chi^2(p)}{p^2}\bigg)
$$
and $\chi$ the non-principal character modulo $4$. 
The constant $C$ agrees with Peyre's prediction \cite[Formule 5.1]{P03}.
\end{theorem}

\begin{rem}
If follows from the arguments in \cite{B98} or \cite{LWZ17} that, at least, 
any $\vartheta<\tfrac{1}{9}$ 
is acceptable in Theorem~\ref{t:main}, and further improvements are possible.   
\end{rem}

\medskip

The Manin-Peyre conjectures for smooth toric varieties were established by Batyrev and Tschinkel in their seminal work \cite{BT98}.  Since our cubic surface $V$ is a (non-split) toric surface, the main term of the asymptotic formula \eqref{N_UB=} can be derived from \cite{BT98}. 
In addition to providing a different proof of the Manin-Peyre's conjectures  for  $V$ 
and  to getting a power-saving error term of the counting function $N_U(B)$, this note also serves to complement the results in \cite{LWZ17},  in which Manin's conjecture for the cubic hypersurfaces $S_n\subset \PP^{n+1} $ defined by the equation
$$x_0^3 = (x_1^2 + \ldots+ x_n^2)  x_{n+1}$$
with $n=4k$ was established. The cubic surface $V$ is the case for $ n=2$.

We conclude the introduction by a brief discussion of  the split toric surface of $\PP_{\QQ}^3$ given by
$$V' :   \hskip 3mm   x_0x_1x_2=x_3^3. $$  
The variety $V'$ is isomorphic to $V$ over $\QQ(i)$ and was well studied by a number of  authors. Manin's conjecture for 
$V'$ is a consequence of 
Batyrev and Tschinkel \cite{BT98}.  Others  include the first author \cite{B98},   the first author and Swinnerton-Dyer  \cite{BSD},  
Fouvry  \cite{Fo}, Heath-Brown and Moroz \cite{HBM}  and Salberger  \cite{Sa}.  Derenthal and Janda \cite{DJ13} established Manin's conjecture for $V'$ over imaginary quadratic fields of class number one and Frei \cite{Fr13}  further generalized their work to arbitrary number fields. 
Of the unconditional  asymptotic formulae obtained, the strongest is the one
in \cite{B98}, which yields the estimate
$$ 
N_U(B)=BP(\log B) + O\big(B^{7/8}\exp(-c(\log B)^{3/5}(\log\log B)^{-1/5})\big),
$$ 
where $U$ is a Zariski  open subset of $V'$, and $P$ is a polynomial of degree~$6$ and $c$ is a positive constant.
In \cite{BSD}, even the second term of the counting function $N_U(B)$ is established  under the Riemann Hypothesis as well as the assumption 
that all the zeros of the Riemann $\zeta$-function are simple.

\section{Geometry and Peyre's constant}


In \cite{P95}, Peyre proposed a general conjecture about the shape of the leading constant arising in the asymptotic formula for the number of points of bounded height but only for smooth Fano varieties.

The surface $V$ that we study in this note is singular so we can not apply directly this conjecture and \cite[D\'efinition 2.1]{P95}. To get around this, we construct explicitly in this section a minimal resolution $\pi:\widetilde{V} \rightarrow V$ of $V$ and show that for $U=V\smallsetminus\{\ell_1\cup\ell_2\cup\ell_3\}$ and $\widetilde{U}=\pi^{-1}(U)$, we have $\pi_{|\widetilde{U}}:\widetilde{U} \cong U$. This implies that our counting problem on $V$ can be seen as a counting problem on the smooth variety $\widetilde{V}$ since
$$
N_U(B)=\#\{\mathbf{x}\in \widetilde{U}(\QQ) : H\circ \pi({\bf x})\leq B\}
$$
where $H \circ \pi$ is an anticanonical height function on $\widetilde{V}$. Indeed, by \cite[Lemma 1.1]{CT88} the surface $V$ has only du Val singularities which are canonical singularities (see \cite[Theorem 4.20]{KM98}) and alluding to \cite[2.26, 4.3, 4.4 and 4.5]{KM98}, we can conclude that $\pi^{\ast}K_{V}=K_{\widetilde{V}}$ where $K_V$ and $K_{\widetilde{V}}$ denote the anticanonical divisors of $V$ and $\widetilde{V}$ respectively. 
However, $\widetilde{V}$ is not a Fano variety and therefore we still can not apply \cite[D\'efinition 2.1]{P95}.

We nevertheless establish in this section that $\widetilde{V}$ is ``almost Fano" in the sense of \cite[Definition 3.1]{P03}. Alluding to the fact that the original conjecture of Peyre has been refined by Batyrev and Tscinkel \cite{BT98b} and Peyre \cite{P03} to this setting, we may refer to \cite[Formule empirique 5.1]{P03} to interpret the constant $C$ arising in our Theorem~\ref{t:main}. According to \cite[Formule empirique 5.1]{P03}, the  
 leading constant $C$ in our Theorem~\ref{t:main} takes the form  
\begin{equation}\label{C=alpha}
C=\alpha(\widetilde{V})\beta(\widetilde{V}) \tau(\widetilde{V}) 
\end{equation}
where $\alpha(\widetilde{V})$ is
a rational number defined in terms of the cone of effective divisors, $\beta(\widetilde{V})$ a cohomological invariant 
and $\tau(\widetilde{V})$ a Tamagawa number. For more details, see d\'efinition 4.8 of \cite{P03}.

Our main strategy to check that the constant $C$ in Theorem \ref{t:main} agrees with the prediction \cite[Formule empirique 5.1]{P03} relies in a crucial way on the (non-split) toric structure of the surface $V$ and on results from \cite{BT98b}.


\subsection{Minimal resolution of $V$ and interpretation of the power of $\log B$}

We refer the reader to the following references for details about toric varieties over arbitrary fields \cite{Oda,Fulton1,Demazure, Danilov} and especially \cite{BT95, BT98} and \cite[End of \S 8]{Sa}.

The toric surface $V$ is easily seen to be an equivariant compactification of the non-split torus $T$ given by the equation $x_0(x_1^2+x_2^2)=1$. The torus $T$ is isomorphic to $R_{\QQ(i)/\QQ}(\mathbb{G}_m)$ where $R_{\QQ(i)/\QQ}(\cdot)$ denotes the Weil restriction functor and is split by the quadratic extension $k=\QQ(i)$. We now introduce $M=\widehat{T}_k:=\mbox{Hom}(T,k^{\times})$ the group of regular $k$-rational characters of $T$ and $N=\mbox{Hom}(M,\mathbb{Z})$. Alluding to \cite[Lemma 1.3.1]{Serre}, we see that $M \cong N \cong \mathbb{Z}\times \mathbb{Z}$ with the Galois group $G=\mbox{Gal}(k/\QQ)\cong \ZZ/2\ZZ$ interchanging the two factors. Let $(e_1,e_2)$ be a $\ZZ$-basis of $N$. In a similar manner as in \cite[Example 11.50]{Sa}, we denote by $\Delta$ the fan of $N_{\RR}=N \otimes \RR$ given by the rays $\rho_1, \rho_2, \rho'_2$ generated by $-e_1-e_2$, $-e_1+2e_2$ and $2e_1-e_2$.

\vspace{1.5cm}
\begin{figure}[h!]
\centering 
\begin{pspicture}(-1.2,-7)
\psgrid[griddots=10, subgriddiv=0, gridlabels=0pt,gridcolor=lightgray](-2,-2)(1,1)
\psline{->}(-1,-1)(-2,-2)
\psline{->}(-1,-1)(-2,1)
\psline{->}(-1,-1)(1,-2)
\rput(-2.15,-2.17){\color{black} $\rho_1$ \color{black}}
\rput(-2.15,1.1){\color{black} $\rho_2$ \color{black}}
\rput(1.15,-2.2){\color{black} $\rho'_2$ \color{black}}
\end{pspicture}
    \vspace{-4.5cm}
    \caption{The fan ${\Delta}$.}
\end{figure}
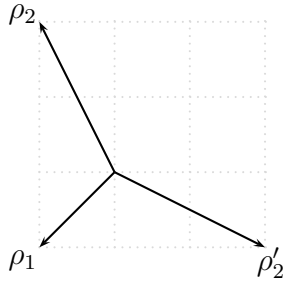

The fan $\Delta$ is $G$-invariant in the sense of \cite[Definition 1.11]{B98} and hence defines a non-split toric surface $P_{\Delta}$ over $\QQ$. Using the same arguments as in \cite[Example 11.50]{Sa}, one easily sees that the $k$-variety $P_{\Delta,k}=P_{\Delta} \otimes_{{\rm Spec}(\QQ)}{\rm Spec}(k)$ is given by the equation $x_3^3=x_0z_1z_2$ with $G$ exchanging $z_1$ and $z_2$. The change of variables $x_1=(z_1+z_2)/2$ and $x_2=(z_1-z_2)/(2i)$ yields that $P_{\Delta,k}$ is isomorphic to the variety of equation $x_3^3=x_0(x_1^2+x_2^2)$, all the variables being $G$-invariant. Hence, the surface $V$ is a complete algebraic variety such that $V\otimes_{{\rm Spec}(\QQ)}{\rm Spec}(k)$ is isomorphic to $P_{\Delta,k}$, the isomorphism being compatible with the $G$-actions. Then, theorem 1.12 of \cite{BT98} allows us to conclude that $V$ is given by the $G$-invariant fan $\Delta$ after noting that the assumption that the fan is regular is not necessary.

The fan $\Delta$ is not complete and regular in the sense of \cite[Definition 1.9]{BT98b} which accounts for the fact that $V$ is singular. As in \cite[Example 11.50]{Sa}, there exists a complete and regular refinement $\tilde{\Delta}$ of $\Delta$ given by the extra rays  $\tilde{\rho}_1, \tilde{\rho}_2, \tilde{\rho}_3, \tilde{\rho}'_1, \tilde{\rho}'_2, \tilde{\rho}'_3$ generated by $-e_1$, $-e_1+e_2$, $e_2$, $-e_2$, $e_1-e_2$ and $e_1$.

\vspace{1.2cm}
\begin{figure}[h!]
\centering 
\begin{pspicture}(-1.2,-7)
\psgrid[griddots=10, subgriddiv=0, gridlabels=0pt,gridcolor=lightgray](-2,-2)(1,1)
\psline{->}(-1,-1)(-2,-2)
\psline{->}(-1,-1)(-2,1)
\psline{->}(-1,-1)(1,-2)
\psline[linecolor=green]{->}(-1,-1)(-2,-1)
\psline[linecolor=green]{->}(-1,-1)(-2,0)
\psline[linecolor=green]{->}(-1,-1)(-1,0)
\psline[linecolor=green]{->}(-1,-1)(-1,-2)
\psline[linecolor=green]{->}(-1,-1)(0,-2)
\psline[linecolor=green]{->}(-1,-1)(0,-1)
\rput(-2.15,-2.2){\color{black} $\rho_1$ \color{black}}
\rput(-2.15,1.1){\color{black} $\rho_2$ \color{black}}
\rput(1.15,-2.2){\color{black} $\rho'_2$ \color{black}}
\rput(-2.15,-1){\color{green} $\tilde{\rho}_1$ \color{black}}
\rput(-2.15,0){\color{green} $\tilde{\rho}_2$ \color{black}}
\rput(-0.9,0.18){\color{green} $\tilde{\rho}_3$ \color{black}}
\rput(-1,-2.2){\color{green} $\tilde{\rho}'_1$ \color{black}}
\rput(0.27,-1){\color{green} $\tilde{\rho}'_3$ \color{black}}
\rput(0.27,-2.2){\color{green} $\tilde{\rho}'_2$ \color{black}}
\end{pspicture}
    \vspace{-4.5cm}
    \caption{The fan $\tilde{\Delta}$.}
\end{figure}
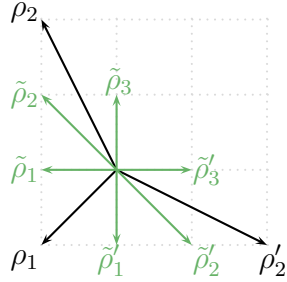

The toric surface $\widetilde{V}$ defined over $\QQ$ by the $G$-invariant fan $\tilde{\Delta}$ is then smooth by \cite[Theorems 1.10 and 1.12]{BT98} and, thanks to \cite[5.5.1]{Danilov} and \cite[\S2.6]{Fulton1}, comes with a proper equivariant birational morphism $\pi: \widetilde{V} \rightarrow V$ which is an isomorphism on the torus $T$. Here $T$ corresponds to the open subset $U=V\smallsetminus\{\ell_1\cup \ell_2\cup \ell_3\}.$ Now the proof of the proposition 11.2.8 of \cite{cox} yields that $\pi$ is a crepant resolution and hence that it is minimal since we are in dimension 2.

We note that thanks to \cite[Corollaire 3]{Demazure} and \cite[Proposition 1.15]{BT98}, the minimal resolution $\widetilde{V}$ is ``almost Fano" in the sense of \cite[Definition 3.1]{P03}.

Let us now turn to the computation of the Picard group of $\widetilde{V}$. To this end, we will exploit the exact sequence given by \cite[Proposition 1.15]{BT98}. With the notations of \cite[Proposition 1.15]{BT98}, we have $M^G \cong \mathbb{Z}$ generated by $e_1^{\ast}+e_2^{\ast}$ if $(e_1^{\ast},e_2^{\ast})$ is a $\ZZ$-basis of $M$. Moreover, a function $\varphi \in \mbox{PL}(\tilde{\Delta})^G$ being completely determined by its integer values on $\rho_1,\rho_2$ and $\tilde{\rho}_1,\tilde{\rho}_2$ and $\tilde{\rho}_3$, we have $\mbox{PL}(\tilde{\Delta})^G\cong \mathbb{Z}^5$. Finally, the $\ZZ$-module $M$ is a permutation module and therefore $H^1(G,M)$ is trivial. Bringing all of that together yields that $\mbox{rk}(\mbox{Pic}(\widetilde{V}))=5-1=4,$ which agrees with the prediction coming from Manin's conjecture regarding the power of $\log B$ in Theorem \ref{t:main}.

\subsection{The factor $\alpha$}

We will use the same method as in \cite[Lemma 5]{BBS} to compute the nef cone volume $\alpha(\widetilde{V})$ and we refer the reader to \cite[Lemma 5]{BBS} for more details and definitions.

Let $T_i$, $T'_i$, $\widetilde{T}_i$ and $\widetilde{T}'_i$ be the Zariski closures of the one dimensional tori corresponding respectively to the cones $\RR_{\geqslant 0} \rho_i$, $\RR_{\geqslant 0} \rho'_i$, $\RR_{\geqslant 0} \tilde{\rho}_i$ and $\RR_{\geqslant 0} \tilde{\rho}'_i$. We also introduce the $G$-invariant divisors
$$
D_1=T_1, \hspace{2mm} D_2=T_2+T'_2, \hspace{2mm} D_3=\widetilde{T}_1+\widetilde{T}'_1, \hspace{2mm} D_4=\widetilde{T}_2+\widetilde{T}'_2, \hspace{2mm} D_5=\widetilde{T}_3+\widetilde{T}'_3.
$$
Using \cite[Proposition 1.15]{BT98}, one immediately sees that $\mbox{Pic}(\widetilde{V})$ is generated by $D_1, D_2, D_3, D_4, D_5$ with the relation $D_5=2D_1+D_2-D_4$ and that the divisor 
$$
D_1+D_2+D_3+D_4+D_5 \sim 3D_1+2D_2+D_3
$$
is an anticanonical divisor for $\widetilde{V}$. Following the strategy of \cite[Lemma 5]{BBS} and using the same notations than in \cite[Lemma 5]{BBS}, it now follows that $C_{{\rm eff}}^{\vee}$ is the subset of $\RR^4_{\geqslant 0}$ given by $2z_1+z_2-z_4 \geqslant 0$ and that $H_{\widetilde{V}}$ is given by the equation $3z_1+2z_2+z_3=1$. Therefore, a straightforward computation finally yields
$$
\begin{aligned}
\alpha(\widetilde{V})&=\int_{0}^1(1-z_3)^2\mbox{d}z_3 \times \frac{1}{2}\mbox{Vol}\left\{ (z_1,z_4) \in \RR^2_{\geqslant 0} : \begin{array}{l}
 3z_1 \leqslant 1,\\
 2z_4-z_1\leqslant 1	
 \end{array}
\right\}\\
&=\frac{1}{6}\int_{z_1=0}^{\frac{1}{3}}\left(\int_{z_4=0}^{\frac{1+z_1}{2}}\mbox{d}z_4\right)\mbox{d}z_1=\frac{7}{216}.
\end{aligned}
$$

%
	\subsection{The factor $\beta$}
Let us now briefly justify that $\beta(\widetilde{V})=1$. We know that $\widetilde{V}$ is birational to the torus $R_{\QQ(i)/\QQ}(\mathbb{G}_m)$. But the open immersion $\mathbb{G}_{m,\QQ(i)} \hookrightarrow \mathbb{A}^2_{\QQ(i)}$ gives rise to an open immersion $R_{\QQ(i)/\QQ}(\mathbb{G}_m) \hookrightarrow \mathbb{A}^2_{\QQ}$ by taking the functor $R_{\QQ(i)/\QQ}(\cdot)$ and by alluding to \cite[Proposition 4.9]{Scha}. Hence, $R_{\QQ(i)/\QQ}(\mathbb{G}_m)$ is rational and so is $\widetilde{V}$. Finally this implies that $\beta(\widetilde{V})=1$ (see \cite[section 5]{BBS} for details).

\subsection{The Tamagawa number}
\subsubsection{Conjectural expression}
Let us choose $S=\{\infty,2\}$ and note that our definition will be independent of that choice. We have from \cite[Theorem 1.3.2]{BT95} that $\mbox{Pic}(\widetilde{V}_{\overline{\QQ}})$ is the free abelian group generated by the divisors $T_1, T_2, T'_2, \widetilde{T}_1, \widetilde{T}'_1, \widetilde{T}_2, \widetilde{T}'_2$ defined in \S 2.2 with the following $G$-action 
$$
\sigma(T_2)=T'_2, \quad \sigma(\widetilde{T}_i)=\widetilde{T}'_i \quad (i \in \{1,2\})
$$
if $\sigma$ denotes the complex conjugation. Alluding to \cite[Defintion 7.1]{Ja}, we have the following conjectural expression
$$ \tau(\widetilde{V}) :=\lim_{s\to 1^+}(s-1)^4L_S(s, \chi_{{\rm Pic}}(\widetilde{V}_{\overline{\QQ}}))\omega_\infty\prod_{p} \lambda_p^{-1}\omega_p
$$ 
where 
$$
L_S(s, \chi_{{\rm Pic}}(\widetilde{V}_{\overline{\QQ}}))=\prod_{p \not \in S}\det\left(\mbox{Id}-p^{-s}\mbox{Frob}_p \mid \mbox{Pic}(\widetilde{V}_{\overline{\QQ}})^{I_p}\right)^{-1}
$$
with $I_p$ the inertia group and $\mbox{Frob}_p$ a representative of the Frobenius automorphism and where
$$
\lambda_p=L_p(1, \chi_{{\rm Pic}}(\widetilde{V}_{\overline{\QQ}})), \quad \omega_{\infty}=\omega_{\infty, \widetilde{V}}(\widetilde{V}(\RR)), \quad \omega_{p}=\omega_{p,\widetilde{V}}(\widetilde{V}(\QQ_p))
$$
for measures $\omega_{v,\widetilde{V}}$ on $\widetilde{V}(\QQ_v)$ whose proper definitions are postponed to the next section (they are the measures $\omega_{\mathcal{K},v}$ defined in \cite[\S 2]{BT98}) and where $\lambda_p$ is taken to be 1 for $p \in S$.

 Let $\Re e(s)>1$. First we notice that for all $p\not \in S$, we have that $I_p$ is trivial since $p$ is not ramified in $\QQ(i)$. Then the Frobenius $\mbox{Frob}_p$ being trivial for all $p \equiv 1 \Mod{4}$, it is easy to see that in that case
$$
\det\left(\mbox{Id}-p^{-s}\mbox{Frob}_p \mid \mbox{Pic}(\widetilde{V}_{\overline{\QQ}})^{I_p}\right)=\left(1-\frac{1}{p^s}\right)^7.
$$
When $p \equiv 3 \Mod{4}$, $\mbox{Frob}_p$ is of order 2 with the same action than $\sigma$ on $\mbox{Pic}(\widetilde{V}_{\overline{\QQ}})$ and hence one sees immediately that
$$
\det\left(\mbox{Id}-p^{-s}\mbox{Frob}_p \mid \mbox{Pic}(\widetilde{V}_{\overline{\QQ}})^{I_p}\right)=\left(1-\frac{1}{p^s}\right)\left(1-\frac{1}{p^{2s}}\right)^3.
$$
Bringing all of this together yields that
$$
L_S(s, \chi_{{\rm Pic}}(\widetilde{V}_{\overline{\QQ}}))=\prod_{p>2}\bigg( 1-\frac{1}{p^s} \bigg)^{-4}\bigg( 1-\frac{\chi(p)}{p^s} \bigg)^{-3}
$$
and 
$$
\lim_{s\to 1}(s-1)^4L_S(s, \chi_{{\rm Pic}}(\widetilde{V}_{\overline{\QQ}}))=\frac{L(1,\chi)^2}{2^4}=\frac{1}{2^4} \times \left(\frac{\pi}{4}\right)^3
$$
and therefore
$$
\tau(\widetilde{V})=\left(\frac{\pi}{4}\right)^3\omega_{\infty} \times \frac{\omega_2}{2^4}\times \prod_{p>2} \bigg( 1-\frac{1}{p} \bigg)^{4}\bigg( 1-\frac{\chi(p)}{p} \bigg)^{3}\omega_p.
$$

\subsubsection{Construction of the Tamagawa measure}
Let us write here $V_{(i)}$ for the affine subset of $V$ where $x_i \neq 0$ with coordinates $x_j^{(i)}=x_j/x_i$ for $j \neq i$. Note that $V_{(i)}$ is defined by the equation 
$$
f^{(i)}(x_0^{(i)},\dots,\widehat{x_i^{(i)}},\dots,x_3^{(i)})=\left(\frac{x_3}{x_i}\right)^3-\frac{x_0}{x_i}\left(\left(\frac{x_1}{x_i}\right)^2+\left(\frac{x_2}{x_i}\right)^2\right)
$$
where $(x_0^{(i)},\dots,\widehat{x_i^{(i)}},\dots,x_3^{(i)})$ denotes $(x_0^{(i)},x_2^{(i)},x_3^{(i)})$ after removing the $i$-th component.
The same arguments as in \cite[\S 13]{FP} go through to yield that $\omega_{V} \cong \mathcal{O}_V(-1)$ and that such an isomorphism is given on $V_{(i)}$ by
$$
x_i^{-1} \longmapsto \frac{(-1)^{i+t}}{\partial f^{(i)}/\partial x_j^{(i)}}\mbox{d}x_k^{(i)}\wedge \mbox{d}x_{\ell}^{(i)}
$$
for $k<\ell \in \{0,1,2,3\}\smallsetminus\{i\}$, $\{i,j,k,\ell\}=\{0,1,2,3\}$ and $t=k+\ell$ if $k<i<\ell$ and $t=k+\ell-1$ otherwise. Moreover, we have already seen that $\omega_{\widetilde{V}}\cong \pi^{\ast}\omega_V$. The dual sections $\tau_i$ of $s_i$ in $\omega^{-1}_V\cong\mathcal{O}_V(1)$ define the embedding $V \hookrightarrow \mathbb{P}^3$ under consideration in this note and the morphism $\widetilde{V} \rightarrow V \hookrightarrow \mathbb{P}^3$ is given by the sections $\pi^{\ast}\tau_i$ of $H^0(\widetilde{V},\omega^{-1}_{\widetilde{V}})$. 

Consider now the subsequent Arakelov heights $(\omega^{-1}_{\widetilde{V}},(||.||_v)_{v \in {\rm Val}(\QQ)})$ and $(\omega^{-1}_{{V}},(||.||'_v)_{v \in {\rm Val}(\QQ)})$ defined by these global sections where for all $v \in \mbox{Val}(\QQ)$, $x\in V(\mathbb{Q}_v)$, $y \in \widetilde{V}(\mathbb{Q}_v)$, $\tau \in \omega^{-1}_{\widetilde{V}}$ and $\sigma \in \omega^{-1}_{{V}}$ we use respectively the $v$-adic metrics defined by
$$
||\tau||_v=\min_{0\leqslant i \leqslant 3 \atop \pi^{\ast}\tau_i \neq 0}\left\{\left| \frac{\tau}{\pi^{\ast}\tau_i(x)}\right|_v\right\}
$$
if $v$ is finite and 
$$
||\tau||_{\infty}=\min\left\{\min_{ i \in \{0,3\} \atop \pi^{\ast}\tau_i \neq 0}\left\{\left|\frac{\tau}{\pi^{\ast}\tau_i(x)}\right|_{\infty}\right\}, \left(\left| \frac{\pi^{\ast}\tau_1(x)}{\tau}\right|_{\infty}^2+\left|\frac{\pi^{\ast}\tau_2(x)}{\tau}\right|_{\infty}^2 \right)^{-\frac{1}{2}}\right\}
$$
if $v$ is the archimedean place and $\tau\neq 0$ and
$$
||\sigma||'_v=\min_{0\leqslant i \leqslant 3 \atop \tau_i \neq 0}\left\{\left| \frac{\sigma}{\tau_i(y)}\right|_v\right\}
$$
if $v$ is finite and
$$
||\sigma||'_{\infty}=\min\left\{\min_{ i \in \{0,3\} \atop \tau_i \neq 0}\left\{\left|\frac{\sigma}{\tau_i(y)}\right|_{\infty}\right\}, \left(\left| \frac{\tau_1(y)}{\sigma}\right|_{\infty}^2+\left|\frac{\tau_2(y)}{\sigma}\right|_{\infty}^2 \right)^{-\frac{1}{2}}\right\}
$$
if $v$ is the archimedean place and $\sigma\neq 0$. These heights correspond to the heights $H$ on $V$ and $H \circ \pi$ on $\widetilde{V}$ that we used in our counting problem. Applying the definition 2.2.1 of \cite{P03}, we get a measure $\omega_{v,\widetilde{V}}$ on $\widetilde{V}(\QQ_v)$ associated to the $v$-adic metric $||.||_v$ which is the measure defined in \cite[\S 2]{BT98} and used in \S 2.4.1 and a measure $\omega_{v,V}$ on $V(\QQ_v)$ associated to the $v$-adic metric $||.||'_v$.

\subsubsection{Computation of the archimedean density}

We follow once again the strategy adopted in \cite[\S 13]{FP}. One sees easily that $U=V_{(3)}$ and the same argument as in \cite[\S 13]{FP} shows that
$$
\omega_{\infty}=\omega_{\infty,\widetilde{V}}(\widetilde{V}(\RR))=\omega_{\infty,\widetilde{V}}(\pi^{-1}(U)(\RR)).
$$
Now the local coordinates $x_1^{(3)}-1$ and $x_2^{(3)}$ at the rational point $(1,1,0)$ of $U$ give an isomorphism $$U(\RR) \cong W=\{(z_1,z_2) \in \RR^2 : (z_1+1)^2+z_2^2 \neq 0\}$$ and a similar computation as in \cite[\S 13]{FP} yields 
$$
\begin{aligned}
\omega_{\infty,\widetilde{V}}(\pi^{-1}(U)(\RR))&=\int_{\RR^2} \frac{\mbox{d}z_1\mbox{d}z_2}{\max\left\{1,z_1^2+z_2^2,(z_1^2+z_2^2)^{3/2}\right\}}\\
&=\int_{z_1^2+z_2^2\leqslant 1} \mbox{d}z_1\mbox{d}z_2+\int_{z_1^2+z_2^2>1} \frac{\mbox{d}z_1\mbox{d}z_2}{(z_1^2+z_2^2)^{3/2}}\\
&=3\pi
\end{aligned}
$$
after a polar change of coordinates. We can therefore conclude that $\omega_{\infty}=3\pi$.

\subsubsection{Computation of $\omega_p$ for odd $p$}
Thanks to the remarks of \cite[Page 187]{Sa}, one can construct a model $\widetilde{\mathcal{V}}$ over $\mbox{Spec}(\ZZ)$ satisfying the conditions of \cite[Notation 4.5]{P03} with $S=\{\infty, 2\}$. Hence, one can consider the reduction $\widetilde{\mathcal{V}}_p$ modulo $p$ of $\widetilde{\mathcal{V}}$ for every prime number $p$. 

The torus $T$ has good reduction $T_p$ for every prime $p>2$ since $p$ is not ramified in $\QQ(i)$ and $T_p$ is a split torus of rank 2 if $p \equiv 1\Mod{4}$ and a non-split torus of rank 2 split by $\mathbb{F}_{p^2}$ if $p \equiv 3\Mod{4}$. Hence, the reduction $\widetilde{\mathcal{V}}_p$ modulo $p$ can be realized as the toric variety over $\mathbb{F}_p$ under the torus $T_p$ given by the fan $\Delta'$ which is invariant under $\mbox{Frob}_p$. Since the fan stays regular and complete, we can conclude that $\widetilde{\mathcal{V}}_p$ is smooth and hence that $\widetilde{\mathcal{V}}$ has good reduction modulo $p>2$ (see \cite{Danilov}). 

We can therefore apply \cite[Corollary 6.7]{Ja} to obtain for all odd $p$ the following expression
$$
\omega_p=\frac{\#  \widetilde{\mathcal{V}} (\FF_p)}{p^2}.
$$
Now alluding to Weil's formula, we obtain
$$
\omega_p=1+\frac{{\rm Tr} ({\rm Frob}_p |{\rm Pic}(\widetilde{V}_{\overline{\QQ}}))}{p}+\frac{1}{p^2}=1+\frac{4+3\chi(p)}{p}+\frac{1}{p^2}
$$
by using the description of the action of $\mbox{Frob}_p$ on ${\rm Pic}(\widetilde{V}_{\overline{\QQ}})$ given in~\S 2.4.1.


\subsubsection{Computation of $\omega_2$}
For $p=2$, the model $\tilde{\mathcal{V}}$ having bad reduction, we appeal to the lemma 6.6 of \cite{Ja} to compute $\omega_2$. By \cite[Proposition 2.10]{BT98} and noting that the smooth assumption is not necessary, one gets
$$
\omega_{2,\widetilde{V}}(\widetilde{V}(\QQ_2))=\omega_{2,\widetilde{V}}(\pi^{-1}(U)), \quad \omega_{2,{V}}({V}(\QQ_2))=\omega_{2,{V}}(U).
$$
Now an analogous computation as the one in \S 2.4.3 yields that both quantities $\omega_{2,\widetilde{V}}(\pi^{-1}(U))$ and $\omega_{2,{V}}(U)$ are equal to the expression
$$
\int_{W}\frac{\mbox{d}z_1\mbox{d}z_2}{\max\left\{1,|z_1^2+z_2^2|_2,|z_1(z_1^2+z_2^2)|_2,|z_2(z_1^2+z_2^2)|_2 \right\}}
$$
with $W=\{(z_1,z_2) \in \QQ_2^2 : (z_1+1)^2+z_2^2 \neq 0\}.$
Therefore, $\omega_2$ is equal to $\omega_{2,{V}}({V}(\QQ_2))$ and \cite[Remark 6.8]{Ja} implies that
$$
\omega_2=\lim_{n\to +\infty} \frac{N(2^n)}{2^{3n}}, 
$$ 
where
$$
N(2^n):=\#\big\{ {\bf x}\bmod{2^n}: \; x_0(x_1^2+x_2^2)\equiv x_3^3 \bmod{2^n}\big\} . 
$$
Let $v_2(x_1^2+x_2^2)=k$ and  $v_2(x_0)=k_0$. If $k=1+2k'$ is odd and $1+2k'<n$, then the number of $(x_1,x_2)$ satisfying $v_2(x_1^2+x_2^2)=k$ is $2^{2n-2k'-2}$. There are $2^{n-(1+2k'+k_0)/3-1}$ ways to choose $x_3$ and then $2^{2k'+1}$ choices for $x_0$. 
Then, in the case where $v_2(x_1^2+x_2^2)$ is odd, the number of solutions is 
asymptotic to  
$$ 
2^{3n} 
\sum_{\substack{3\mid 1+2k'+k_0'\\ k'\geq 0}} 
2^{-2-(1+2k'+k_0')/3}\sim \frac{5}{6}2^{3n}. 
$$
The number of $(x_1,x_2)$ satisfying $v_2(x_1^2+x_2^2)=2k'$ is, at least for $2k'<n$, 
equal to $2^{2n-2k'-1}.$ 
There are $2^{n-(2k'+k_0)/3-1}$ ways to choose $x_3$ and then $2^{2k'}$ choices for $x_0$. 
Summing over $3\mid 2k'+k_0$ and $k'\geq 0$ we get the contribution of the case $ v_2(x_1^2+x_2^2)$ even in $N(2^n)$, which is asymptotic to 
$\frac{7}{6}\cdot 2^{3n}.$ It follows that  
$$
\omega_2=2=1+\frac{2+3\chi(2)+2\chi^2(2)}{2}+\frac{\chi^2(2)}{2^2}. 
$$


\subsubsection{Conclusion} Bringing everything together yields the following expression for the Peyre constant
$$
\alpha(\widetilde{V})\beta(\widetilde{V})\tau(\widetilde{V})
=\frac{7}{216} (3\pi)\bigg(\frac{\pi}{4}\bigg)^3\tau.
$$
This is in agreement with the constant $C$ in \eqref{TmC=}.

\section{Proof of Theorem~\ref{t:main}}
By symmetry, we have
$$
N_U(B)
=\#\bigg\{
{\bf x}\in E  : x_0(x_1^2+x_2^2)=x_3^3, \;  
\max\Big\{ x_0,\sqrt{x_1^2+x_2^2}\Big\}\leq B \bigg\}, 
$$
where 
$E:=\{{\bf x}\in \NN\times \ZZ^2\times \NN : \gcd(x_0, x_1, x_2, x_3)=1\}$ 
and 
$\NN=\ZZ_{\geq 1}$.  
As in \cite{B98}, we parametrize  $x_1^2+x_2^2$, $x_0$ and $ x_3$ by 
$$
x_1^2+x_2^2=n_1n_2^2n_3^3,\quad 
x_0=n_1^2n_2n_4^3, \quad 
x_3=n_1 n_2n_3n_4 , 
$$ 
where $n_1$ and $n_2$ are squarefree and $\gcd (n_1,n_2)=1$ 
which is equivalent to 
$\mu^2(n_1n_2)=1$. It follows that  
$$
N_U(B)=4\sum_{\substack{{\bf n}\in \NN^4\\ \mu^2(n_1n_2)=1\\ n_1^2n_2n_4^3\leq B\\ n_1n_2^2n_3^3\leq B^2}}r(n_1n_2^2n_3^3,n_1n_2n_4)
$$
where 
$$
r(n,m):=\tfrac  14\#\big\{ (x_1,x_2)\in \ZZ^2 : x_1^2+x_2^2=n,\; ((x_1,x_2),m)=1\big\}. 
$$

Here, we remark that our choice of height function is particularly well suited to handle the expression $r(n,m)$.

Let $\chi$ be the non-principal character modulo $4$ and $r_0:=1*\chi$.
The quantity $r(n,m)$ is a multiplicative arithmetic function in $n$, and we have
$$r(n,m):=\prod_p r\big(p^{v_p(n)}, p^{v_p(m)}\big).$$
We use  the fact that, when $\nu\geq 1$, 
$$r(p^\nu,p)=\begin{cases}2& \mbox{if   $p\equiv 1\bmod 4$},\\
0& \mbox{if   $p\equiv 3\bmod 4$},
\\
1& \mbox{if   $\nu=1$, $p=2$},
\\
0& \mbox{if   $\nu\geq 2$, $p=2$}.\end{cases}$$
Then, when $\nu_1+\nu_2\leq 1$, the value of $r(p^{\nu_1+2\nu_2+3\nu_3}, p^{\nu_1+ \nu_2+ \nu_4})$ is given by
$$ \begin{cases}
1 & \mbox{if $(\nu_1,\nu_2,\nu_3,\nu_4)=(0,0,0,\nu_4)$,}\\
r_0(p^{3\nu_3}) & \mbox{if $(\nu_1,\nu_2,\nu_3,\nu_4)=(0,0,\nu_3,0)$,} \\
0& \mbox{if $(\nu_1,\nu_2)=(0,0)$, $\min \{ \nu_3,\nu_4\}\geq 1$, $p\equiv 2,3\bmod 4$},\\
2& \mbox{if $(\nu_1,\nu_2)=(0,0)$, $\min \{ \nu_3,\nu_4\}\geq 1$, $p\equiv 1\bmod 4$},\\
0& \mbox{if  $\min \{ \nu_1,\nu_3\}\geq 1$, $p\equiv 2,3 \bmod 4$},\\
2& \mbox{if  $\min \{ \nu_1,\nu_3\}\geq 1$, $p\equiv 1 \bmod 4$},\\
0& \mbox{if  $ \nu_1= 1$, $\nu_3=0$, $p\equiv 3  \bmod 4$},\\
2& \mbox{if  $ \nu_1= 1$, $\nu_3=0$, $p\equiv 1 \bmod 4$},\\
1& \mbox{if  $ \nu_1= 1$, $\nu_3=0$, $p=2$},\\
0& \mbox{if  $ \nu_2= 1$, $p\equiv 2,3 \bmod 4$},\\
2& \mbox{if  $\nu_2= 1$, $p\equiv 1 \bmod 4$}.\\
\end{cases}$$

The Dirichlet series associated to this counting problem
is
$$F(s_1,s_2):=\sum_{\substack{{\bf n}\in \NN^4\\ \mu^2(n_1n_2)=1}}\frac{r(n_1n_2^2n_3^3,n_1n_2n_4)}{n_1^{2s_1+s_2}n_2^{s_1+2s_2} n_3^{3s_2}n_4^{3s_1}}, \quad \left(\Re e(s_1),\Re e(s_2)>\frac{1}{3}\right).$$
It can be written as an Euler product of $F_p(s_1,s_2)$, where
$$
F_2(s_1,s_2)=\frac{1}{1-2^{-3s_1}}+\frac{1}{2^{3s_2}-1}+\frac{1}{2^{2s_1+s_2}(1-2^{-3s_1})},
$$
$$
F_p(s_1,s_2)
=\frac{1}{1-p^{-3s_1}}+\frac{1}{p^{6s_2}-1}, 
$$ 
if $p\equiv 3\bmod 4$ and
\begin{align*} 
F_p(s_1,s_2)
=&\frac{1}{1-p^{-3s_1}}+\frac{4-p^{-3s_2}}{p^{3s_2}(1-p^{-3s_2})^2} \cr 
&+2\frac{p^{-3( s_1+s_2)}+p^{-(2s_1+s_2)}+p^{-(s_1+2s_2)}}{(1-p^{-3 s_2} )(1-p^{-3 s_1})} ,
\end{align*}
if $p\equiv 1\bmod 4$.
For $\mathfrak{R} e(s)>1$, let
\begin{align*} 
\zeta_{\QQ(i)}(s) 
&:=\sum_{n\geq 1}\frac{r_0(n)}{ n^s}=\zeta(s)L(s,\chi)\\ 
&=\frac{1}{ 1-2^{-s}}\prod_{p\equiv 3\bmod 4}\frac{1}{ 1-p^{-2s}} 
\prod_{p\equiv 1\bmod 4}\frac{1}{( 1-p^{-s})^2}. 
\end{align*}
Let ${\bf s}$ stand for the pair $(s_1, s_2)$.  Then    
there exists $G$ such that 
$$
F({\bf s})=\zeta (3s_1)\zeta_{\QQ(i)}(3s_2)^2\zeta_{\QQ(i)}(s_1+2s_2)\zeta_{\QQ(i)}(s_2+2s_1)G({\bf s}). 
$$ 
The above quantity $G({\bf s})$ can be written as an Euler product of $G_p({\bf s})$ where 
$$
G_2(1/3, 1/3)=2^{-3}  
$$
while, for $p\equiv 3\bmod 4$ 
$$
G_p({\bf s})= \big(1-p^{-2(s_1+2s_2)}\big)^2\big(1-p^{-2(2s_1+s_2)}\big)
\big(1-p^{-3(s_1+2s_2)}\big). 
$$  
and for $p\equiv 1\bmod 4$, 
\begin{align*} G_p({\bf s})
=& \big(1-p^{-3s_2}\big)^4 
\big(1-p^{-(s_1+2s_2)}\big)^2
\big(1-p^{-(2s_1+s_2)}\big)^2
\\
&+\big(1-p^{-3s_1}\big) 
\big(4p^{-3s_2}- p^{-6s_2}\big)  \\ 
& \quad \times\big(1-p^{-3s_2}\big)^2
\big(1-p^{-(s_1+2s_2)}\big)^2 \big(1-p^{-(2s_1+s_2)}\big)^2\\ 
&+2\big(p^{-3( s_1+s_2)}+p^{-(2s_1+s_2)}+p^{-(s_1+2s_2)})\\ 
&\quad \times 
\big(1-p^{-3s_2}\big)^3 \big(1-p^{-(s_1+2s_2)}\big)^2 \big(1-p^{-(2s_1+s_2)}\big)^2. 
\end{align*}
The series $F$ is absolutely convergent when $\Re e (s_1)>\tfrac 13  $ and $\Re e (s_2)>\tfrac 13$ and the function $G$ can be analytically continued to
$\Re e (s_1)>\tfrac 16  $ and $\Re e (s_2)>\tfrac 16 $. 
Moreover, we have
\begin{eqnarray}\label{prodG}
G\bigg(\frac 13,\frac 13\bigg) 
&=&\frac{1}{2^3} \prod_{p\neq 2}\bigg( 1-\frac{1}{p}\bigg)^4 
\bigg( 1-\frac{\chi(p)}{p}\bigg)^3  \bigg(1+\frac{4+3\chi(p)}{p}+\frac{1}{p^2}\bigg) 
\nonumber\\ 
&=&\tau.
\end{eqnarray}
Thus $F$ satisfies the assumptions of Theorem~1 of \cite{B01} with $(\beta_1,\beta_2) =(1,2)$, $(\alpha_1,\alpha_2) =(\tfrac 13,\tfrac 13)$, 
\begin{align*}
\ell_1({\bf s})=3s_1, \quad \ell_2({\bf s})=\ell_3({\bf s})=3s_2,\\ 
\ell_4({\bf s})=s_1+2s_2, \quad \ell_5({\bf s})=2s_1+s_2.
\end{align*}
It follows that there exists a constant $\vartheta>0$ and a polynomial  
$Q\in \RR[X]$ of degree $3$ such that 
$$
N_U(B)=BQ(\log B)+O(B^{1-\vartheta}). 
$$
Now alluding to Theorem~2 of \cite{B01} to get the  leading coefficient $C$  of $Q$, 
we obtain
\begin{align*}Q(\log B)&\underset{B \to +\infty}{\sim} \frac{4L(1,\chi)^4G(\tfrac 13,\tfrac 13)}{B}\int_{\substack{(y_1,y_2,y_3,y_4,y_5)\in [1,+\infty[^5\\ y_1^3y_4y_5^2\leq B ,\, y_2^3y_3^3y_4^2y_5\leq B^2 }}\d {\bf y}\\
&\underset{B \to +\infty}{\sim} 4 \bigg(\frac{\pi}{4}\bigg)^4 
G\Big(\frac13,\frac13\Big)
\int_{\substack{( y_3,y_4,y_5)\in [1,+\infty[^3\\  y_4y_5^2\leq B ,\,  y_3^3y_4^2y_5\leq B^2 }}\frac{\d {\bf y}}{y_3y_4y_5}
\\
&\underset{B \to +\infty}{\sim}\frac{\pi^4}{2^6}G\Big(\frac13,\frac13\Big) (\log B)^3{I},
\end{align*}
where 
\begin{align*}{I}
:=& {\rm vol}\big\{( t_3,t_4,t_5)\in \RR_+^3\,:\,  t_4+2t_5 \leq 1 ,\,  3t_3+2t_4 +t_5\leq 2  \big\}.\end{align*}
An straightforward computation immediately yields
\begin{align*} 
{I}
&= \frac13
\int_{0}^{1/2} \int_{0}^{ 1-2t_5 } (2-2t_4-t_5)\d t_4\d t_5 =\frac {7}{72}, 
\end{align*}
and therefore the leading coefficient $C$ of $Q$ is given by
$$
C=\frac {7}{216}\Big(\frac{\pi}{4}\Big)^3(3\pi)G\Big(\frac13,\frac13\Big). 
$$  
By \eqref{prodG} we have $G(\tfrac 13,\tfrac 13)=\tau$, from which \eqref{TmC=} follows. 
This completes the proof. 
\medskip

\section{The descent argument}

Our main argument in order to derive Theorem~\ref{t:main} in section 3 consists of a descent from our original variety $\tilde{V}$ onto the variety of equation 
$$
x_1^2+x_2^2=n_1n_2^2n_3^3.
$$ 
Although this is not required to verify Peyre's conjecture since $\tilde{V}$ is a rational variety, it is particularly interesting to find out which torsor were used during this descent argument because $\tilde{V}$ is a non-split variety. Indeed, as mentioned in \cite{DP}, versal torsors parametrizations (see \cite{Coll} for precise definitions) are mostly used in the case of split varieties and the question of the right approach in the case of non-split varieties is quite natural. Using the Cox ring machinery over nonclosed fields developed in \cite{DP}, all known examples of Manin's conjecture in the case of non-split varieties derived by means of a descent rely on a descent on quasi-versal torsors in the sense of \cite{Coll}. For example, the descent in \cite{Dest} is a descent on torsors of injective type $\mbox{Pic}(V_{\mathbb{Q}(i)})\hookrightarrow \mbox{Pic}(V_{\overline{\mathbb{Q}}})$ whereas it is shown in \cite{DP} that the \textit{ad hoc} descent used in \cite{BB} is a descent on the torsor of injective type $\mbox{Pic}(V)\hookrightarrow \mbox{Pic}(V_{\overline{\mathbb{Q}}})$. Here, we now show in the following lemma that the descent corresponds to a torsor of a different type, which is not quasi-versal. 

With the notations of \S 2.2, we set $\hat{T}=[D_1]\mathbb{Z}\oplus [D_3]\mathbb{Z}\oplus [D_4] \mathbb{Z}$ and $\lambda: \hat{T} \hookrightarrow \mbox{Piv}(\tilde{V}_{\overline{\mathbb{Q}}})$ be the natural embedding.

\begin{lemma}
	Every Cox ring of injective type $\lambda$ is isomorphic to the $\mathbb{Q}$-algebra
	$$
	R=\mathbb{Q}[x_1,x_2,\eta_1,\eta_2,\eta_3,\eta_4]/\left( x_1^2+x_2^2-\eta_2\eta_3^2\eta_4^3\right).
	$$
\end{lemma}

\begin{proof}
	The proof is very similar to the one in \cite[Proposition 2.71]{Pie} and that is why we will not repeat all the details here. Since $\tilde{V}_{\overline{\mathbb{Q}}}$ is a split toric variety, we know by \cite{Sa} that a Cox ring of identity type for $\tilde{V}_{\overline{\mathbb{Q}}}$ is given by
	$$
	\overline{\mathcal{R}}=\overline{\mathbb{Q}}[t_1,t_2,t'_2,\tilde{t}_1,\tilde{t}'_1,\tilde{t}_2,\tilde{t}'_2,\tilde{t}_3,\tilde{t}'_3]
	$$
	where $t_i=\mbox{div}(T_i)$, $t'_i=\mbox{div}(T'_i)$, $\tilde{t}_i=\mbox{div}(\tilde{T}_i)$ and $\tilde{t}'_i=\mbox{div}(\tilde{T}'_i)$. We then have by \cite[Remark 2.51]{Pie} that every Cox ring of injective type $\lambda$ is isomorphic to the ring of invariant of
	$$
	\bigoplus_{m \in \hat{T}}\overline{R}_m
	$$
	where $\overline{R}_m$ is the vector space generated by the degree $m$ elements of $\overline{\mathcal{R}}$. For $m \in \hat{T}$ given by
	$
	m=\left[a_1 D_1+a_3D_3+a_4D_4\right],
	$
	we have to solve the following linear system with $e_i,e'_i,\tilde{e}_i,\tilde{e}'_i\geqslant 0$ to determine $\overline{R}_m$ 
	$$
	\begin{aligned}
	&\left[e_1T_1+e_2T_2+e'_2T'_2+\tilde{e}_1\tilde{T}_1+\tilde{e}'_1\tilde{T}'_1+\tilde{e}_2\tilde{T}_2+\tilde{e}'_2\tilde{T}'_2+\tilde{e}_3\tilde{T}_3+\tilde{e}'_3\tilde{T}'_3 \right]\\
	&\qquad  \qquad \qquad \qquad \qquad \qquad \qquad \qquad \qquad \quad =\left[a_1 D_1+a_3D_3+a_4D_4\right].
	\end{aligned}
	$$
Alluding to the fan $\Delta'$ and \cite[Proposition 1.15]{BT98}, we get that this linear system is equivalent to
	$$
	\left\{
\begin{array}{l}
	\tilde{e}'_3+\tilde{e}_1=\tilde{e}_3+\tilde{e}'_1\\
	\tilde{e}'_2+\tilde{e}_3-\tilde{e}'_3=\tilde{e}_2+\tilde{e}'_3-\tilde{e}_3\\
	e_2+\tilde{e}'_3-2\tilde{e}_3=e'_2+\tilde{e}_3-2\tilde{e}'_3=0.
\end{array}	
	\right.
	$$
	This easily yields that $\overline{R}$ is generated by
	$$
	\eta_1=t_1, \quad \eta_2=\tilde{t}_1\tilde{t}'_1, \quad \eta_3=\tilde{t}_2\tilde{t}'_2, \quad \eta_4=t_2t'_2\tilde{t}_3\tilde{t}'_3, \quad \eta_5=\tilde{t}_1{\tilde{t}_2}^2t_2^3{\tilde{t}_3}^2{\tilde{t}'}_3,
	$$
	and $\overline{\eta}_5$ the conjugate of $\eta_5$ with the relation
	$$
	\eta_5\overline{\eta}_5=\eta_2\eta_3^2\eta_4^3.
	$$
	Using the Galois invariant variables 
	$$
	x_1=\frac{\eta_5+\overline{\eta}_5}{2}, \quad x_2=\frac{\eta_5-\overline{\eta}_5}{2i}
	$$
	one finally ensures that every Cox ring of injective type $\lambda$ is isomorphic to~$R$.
\end{proof}

\begin{ack}{The authors are grateful to thank Ulrich Derenthal and Marta Pieropan for useful discussions. This article was initiated during a 
workshop on number theory held at the Weihai Campus of 
Shandong University in July 2017. The first, third and fourth authors are 
partially supported by the program  
PRC 1457 - AuForDiP (CNRS-NSFC). The third author is supported by the National Science Foundation of China under Grant 11531008, the Ministry of Education of China under Grant IRT16R43, and the Taishan Scholar Project of Shandong Province. The hospitality and financial support of these
institutions are gratefully acknowledged.    
}\end{ack}

\end{document}